\documentclass[letter]{amsart}
\usepackage{amsmath, amsthm, amssymb}
\usepackage{hyperref}
\usepackage{cleveref}
\newtheorem{prop}{Proposition}
\newtheorem{lemma}[prop]{Lemma}
\newtheorem{theorem}[prop]{Theorem}

\theoremstyle{definition}
\newtheorem*{rem}{Remark}
\newtheorem*{ex}{Example}
\newcommand{\N}{\mathbb{N}}
\newcommand{\Z}{\mathbb{Z}}

\newcommand{\F}{\mathbb{F}}
\newcommand{\m}{\mathfrak{m}}
\newcommand{\mSpec}{\operatorname{mSpec}}
\newcommand{\ch}{\operatorname{char}}
\newcommand{\rad}{\operatorname{Rad}}
\newcommand{\Idn}{\text{id}_n}
\DeclareMathOperator{\sgn}{sgn}
\DeclareMathOperator{\diag}{diag}

\everymath{\displaystyle}

\begin{document}
\title{Invertible sums of matrices}
\author{Justin Chen}
\address{Department of Mathematics, University of California, Berkeley,
California, 94720 U.S.A}
\email{jchen@math.berkeley.edu}


\begin{abstract}
We give an elementary proof of a Caratheodory-type result on the invertibility 
of a sum of matrices, due first to Facchini and Barioli. 
The proof yields a polynomial identity, expressing the determinant of a large 
sum of matrices in terms of determinants of smaller sums. Interpreting these
results over an arbitrary commutative ring gives a stabilization result for a 
filtered family of ideals of determinants. Generalizing in another direction
gives a characterization of local rings. 
An analogous result for semilocal rings is also given -- interestingly, 
the semilocal case reduces to the case of matrices.
\end{abstract}


\maketitle

Let $k$ be a field, and let $A_1, \ldots, A_m \in M_n(k)$ be square 
matrices over $k$, for fixed $n, m \in \N$. Write $[m]$ for the index set 
$\{1, \ldots, m\}$, and for a subset $S \subseteq [m]$, write $|S|$ 
for the cardinality of $S$.
\begin{theorem}
If $A_1 + \ldots + A_m$ is
invertible, then there exists $S \subseteq [m]$
with $|S| \le n$ such that $\sum_{i \in S} A_i$ is invertible.
\end{theorem}

Notice that the upper bound $p \le n$ is the best possible: for the 
$n$ elementary matrices $E_{ii}$, $i = 1, \ldots, n$, $E_{11} + \ldots
+ E_{nn} = \Idn$ is invertible, but any sum of $\le n -1$
of the $E_{ii}$ is not invertible. 

In proving Theorem 1, we may assume 
$m > n$ (if $m \le n$, then take $S = [m]$), which 
we do henceforth. With this, Theorem 1 is then a consequence of 
the following polynomial identity:

\begin{lemma}
Let $T := k[x_{i \beta \gamma} \mid 1 \le i \le m, 1 \le \beta, \gamma \le n]$ 
be a polynomial ring over $k$ in $mn^2$ indeterminates, and let 
$M_i := (x_{i \beta \gamma})_{\beta, \gamma =1}^n \in 
M_n(T)$ be generic matrices, for $i = 1, \ldots, m$. If $m > n$, then as 
polynomials (i.e. elements of $T$), 
\[
\sum_{S \subseteq [m]} (-1)^{|S|} \det \left( \sum_{i \in S} M_i \right) = 0.
\]
In particular,
\[
\det \bigg( \sum_{i=1}^m M_i \bigg) \in \left( \det \bigg( 
\sum_{i \in S} M_i \bigg) \; \Big| \; S \subseteq [m], |S| \le n \right).
\]
\end{lemma}

\begin{proof}[Lemma 2 implies Theorem 1:]
Suppose there exist $A_1, \ldots, A_m \in M_n(k)$ with $A_1 + 
\ldots + A_m$ invertible but every sum of $\le n$ of the $A_i$'s is not 
invertible. Write $A_i =: (a_{i \beta \gamma})_{\beta, \gamma = 1}^n$ 
for $i = 1, \ldots, m$ (so $a_{i \beta \gamma} \in k$). Let
$\varphi : T \to k$, $\varphi(x_{i \beta \gamma}) = 
a_{i \beta \gamma}$ be the evaluation map, inducing $\widetilde{\varphi} :
M_n(T) \to M_n(k)$, $\widetilde{\varphi}(M_i) = A_i$. 
Since $\det$ is a polynomial in the entries of a matrix, 
$\varphi(\det( \sum_{i \in S} M_i)) = \det( \widetilde{\varphi}( \sum_{i \in S} M_i))
= \det(\sum_{i \in S} A_i)$ for every 
$S \subseteq [m]$. Applying $\varphi$ to the containment in Lemma 2 
implies that the nonzero element $\det \bigg( \sum_{i=1}^m A_i \bigg)$
is contained in the ideal $\left( \det \bigg( \sum_{i \in S} A_i \bigg) \; \Big| \; 
S \subseteq [m], |S| \le n \right)$ of $k$, but each generator is 0, a 
contradiction. Thus no such $A_i$'s can exist.
\end{proof}

The polynomial identity of Lemma 2 follows in turn from a 
combinatorial identity:

\begin{lemma}
For $m, n \in \N$, $m > n$, let $\{z_{i,j} \mid 1 \le i \le m, 1 \le j \le n\}$ be a 
set of $mn$ commuting indeterminates. Then (as polynomials in 
$\Z[z_{ij}]$) 
\[
\sum_{S \subseteq [m]} (-1)^{|S|} \prod_{j=1}^n \sum_{i \in S} z_{i,j} = 0.
\]
\end{lemma}

\begin{proof}[Lemma 3 implies Lemma 2:]
By definition, for an $n \times n$ matrix 
$Y = (y_{\beta \gamma})_{\beta, \gamma = 1}^n$,
\[
\det Y = \sum_{\sigma \in S_n} \sgn(\sigma) \prod_{j=1}^n y_{j \sigma(j)}.
\]
Hence
\begin{align*}
\sum_{S \subseteq [m]} (-1)^{|S|} \det \bigg( \sum_{i \in S} M_i \bigg)
= &\sum_{S \subseteq [m]} (-1)^{|S|} \sum_{\sigma \in S_n} \sgn(\sigma) 
\prod_{j=1}^n \sum_{i \in S} x_{i j \sigma(j)} \\
= &\sum_{\sigma \in S_n} \sgn(\sigma) \sum_{S \subseteq [m]} (-1)^{|S|} 
\prod_{j=1}^n \sum_{i \in S} x_{i j \sigma(j)}
\end{align*}

so for fixed $\sigma \in S_n$, setting $z_{i,j} := x_{i j \sigma(j)}$ in Lemma 3 
gives the desired vanishing.

The second statement of Lemma 2 follows from the first, by induction on $m$:
the identity yields $\det \bigg( \sum_{i=1}^m M_i \bigg) \in \left( \det \bigg( 
\sum_{i \in S} M_i \bigg) \; \Big| \; S \subseteq [m], |S| < m \right)$, which 
implies both the base case (by taking $m = n + 1$) and the inductive step.
\end{proof}

\begin{proof}[Proof of Lemma 3:]
As every monomial in the sum is of the form
$z_{i_1,1} \ldots z_{i_n, n}$ for some $i_j \in [m]$, it suffices 
to show that the coefficient of $z_{i_1,1} \ldots z_{i_n,n}$ is $0$, for any 
fixed choice of $i_1, \ldots, i_n \in [m]$. Now $z_{i_1,1} \ldots z_{i_n,n}$ 
appears for a particular $S \subseteq [m]$ iff 
$\{i_1, \ldots, i_n \} \subseteq S$, and for such an $S$, 
$z_{i_1,1} \ldots z_{i_n,n}$ appears exactly once, with coefficient 
$(-1)^{|S|}$. Thus the coefficient of $z_{i_1,1} \ldots z_{i_n,n}$ is 
\begin{align*}
\sum_{\{i_1,\ldots,i_n \} \subseteq S \subseteq [m]} (-1)^{|S|} &=
\sum_{S' \subseteq [m] \setminus \{i_1,\ldots,i_n \}} (-1)^{|S'|+n} \\
&= \sum_{l=0}^{m-n} {m - n \choose l} (-1)^{l+n} \\
&= (-1)^n(1 - 1)^{m-n} = 0. \qedhere
\end{align*}
\end{proof}

\begin{rem}
It has come to our attention that the statement of Theorem 1 has in fact 
appeared before, posed as a problem with accompanying solution in 
\cite{FBCC}. Also, a version of the first statement of Lemma 2 can 
be found in \cite[Theorem 2.2]{CS}. The solution given in \cite{FBCC}
follows a slightly different approach than the proof given here, as well as 
using different lemmas. We have chosen to present the reasoning here 
for its originality, and to emphasize the simple yet pleasing proof of Lemma 3.

It should be noted that the theorems of both \cite{CS} and \cite{FBCC}
are stated only for fields. To the best of our knowledge, the ideal-theoretic 
results below have not been observed before. 
Other studies of determinants (and characteristic polynomials) of sums 
of matrices can be found in \cite{SA}, \cite{RS}. 
\end{rem}

Now let $R$ be a ring (henceforth all rings, except for matrix rings, 
are always commutative with $1 \ne 0$). It is natural to ask: to what extent
do the above results generalize to $M_n(R)$? We first give a generalization of 
Lemma 2:

\begin{prop}
Let $R$ be a ring, $n \in \N$, $I$ an index set (possibly infinite). For 
any collection of matrices $\{ A_i \mid i \in I \} \subseteq M_n(R)$, consider 
the $R$-ideals
\[
I_j := \left( \det \bigg( \sum_{i \in S} A_i \bigg) \; \Big| \; S \subseteq I, 
|S| \le j \right)
\]
for $j \in \N$. Then $0 = I_0 \subseteq I_1 \subseteq \ldots \subseteq I_n
= I_{n+1} = \ldots$ is an ascending chain of ideals which stabilizes at 
position $n$.
\end{prop}

\begin{proof}
It is immediate from the definition that $I_j \subseteq I_{j+1}$ for all $j \in \N$,
so it suffices to show that $I_m \subseteq I_n$ for $m > n$ by induction. 
This follows from the proof of Lemma 2: since the identity in Lemma 2 only
involves coefficients of $\pm 1$, it continues to hold in the polynomial ring
$\Z[x_{i \beta \gamma}]$. Applying the universal map $\Z \to R$ shows that
the identity holds also in $R[x_{i \beta \gamma}]$, and specializing to $R$
gives the result.
\end{proof}

\begin{theorem}
The following are equivalent for a ring $R$:

\indent (i) $R$ is local, i.e. has a unique maximal ideal $\m$

\indent (ii) For all (equivalently some) $n \ge 1$ and $A_1, \ldots, A_m \in 
M_n(R)$ with $A_1 + \ldots + A_m$ invertible, there exists $S \subseteq 
[m]$ with $|S| \le n$ such that $\sum_{i \in S} A_i$ is invertible.
\end{theorem}

\begin{proof}
(i) $\implies$ (ii): First, note that for any $A \in M_n(R)$, $A$ is invertible 
iff $\det A$ is a unit in $R$. Form the ideals $I_j$ for $j \in [m]$ as above.
By hypothesis, $I_m$ contains a unit, so by Proposition 4 so does $I_n$. 
Then one of the generators of $I_n$ is a unit: if not, then each generator
would be in $\m$, hence $I_n$ would be as well, a contradiction. 
(Alternate proof: applying $R \twoheadrightarrow R/\m$ reduces to Theorem 1).

(ii) $\implies$ (i): Suppose $R$ has two distinct maximal ideals $\m_1, \m_2$.
Then there exists $m_1 \in \m_1$, $m_2 \in \m_2$ with $m_1 + m_2 = 1$. For 
an $n$ such that (ii) holds, let $A_i := \diag(0, \ldots, m_1, \ldots, 0) \in M_n(R)$ 
be the diagonal matrix with $m_1$ in the $i^\text{th}$ spot and 0 elsewhere, for 
$i = 1, \ldots, n$, and $A_{n+1} := m_2 \cdot \Idn$. Then $A_1 + \ldots + A_{n+1} 
= \Idn$, but any sum of $\le n$ of the $A_i$'s has determinant either 0, $m_1^n$, 
or $m_2^j$ for some $1 \le j \le n$, hence is not invertible.
\end{proof}

\begin{ex}
(1): Proposition 4 implies that if $I_n \subseteq J$ for some $R$-ideal $J$, then
so is $I_m$ for all $m$; e.g. if $A_1, A_2, A_3, A_4 \in M_3(\Z)$, and
any sum of at most 3 has determinant divisible by 10, then 
$\det(A_1 + A_2 + A_3 + A_4)$ is also divisible by 10. In particular, taking 
$J = I_n$ generalizes Theorem 5: for any ring $R$ and 
$A_1, \ldots, A_m \in M_n(R)$, if $I_n \ne R$ then no (distinct) sum of the $A_i$'s 
is invertible.

(2): For generic matrices, the ideals $I_j$ quickly become infeasible to compute.
The smallest nontrivial case is $n = 2$: here $R = k[x_{i \beta \gamma} \mid 1 \le i \le 3, 
1 \le \beta, \gamma \le 2]$ is a polynomial ring in 12 variables over a field $k$, and 
$A_i = (x_{i \beta \gamma}) \in M_2(R)$, $i = 1, 2, 3$. Then $I_1$ is a complete 
intersection prime ideal of codimension 3, whereas $I_2$ has two minimal 
primes of codimension 5, and one embedded prime of codimension 7.
\end{ex}

Theorem 5 says, in some sense, that the question of when a large sum of matrices
is invertible is determined by the $1 \times 1$ case (precisely) when the ring is local.
Motivated by this, we now shift perspectives and ask: if the ring is semilocal (i.e. has 
only finitely many maximal ideals), when is a large sum of ring elements ($= 1 \times
1$ matrices) invertible? The following result is interesting in that it follows from a 
result for (a specific class of) noncommutative rings, but the proof is 
not obtained by imposing commutativity verbatim!

\begin{theorem}
Let $R$ be a ring with $n$ maximal ideals (say $\mSpec R = \{\m_1, \ldots, \m_n\}$),
and let $a_1, \ldots, a_m \in R$ with $a_1 + \ldots + a_m \in R^\times$.
If $\ch R/\m_i = \ch R/\m_j$ for all $i, j$ (e.g. if $R$ contains a field), then there exists 
$S \subseteq [m]$ with $|S| \le n$ such that $\sum_{i \in S} a_i \in R^\times$.
\end{theorem}

\begin{proof}
For any ring $S$ and $a \in S$, $a \in S^\times$ iff $\bar{a} \in (S/\rad(S))^\times$,
where $\rad(S)$ is the Jacobson radical of $S$, i.e. the intersection of all maximal ideals
of $S$. For $R$ as above, $\rad(R) = \m_1 \cap \ldots \cap \m_n$ is a finite intersection, 
so by Chinese Remainder $R/\rad(R) \cong R/\m_1 \times \ldots \times R/\m_n$. Thus we 
may assume $R$ is a direct product of $n$ fields $k_1, \ldots, k_n$. 

Now by assumption 
$\ch k_i = \ch k_j$ for all $i, j$, so there exists a large field $K$ such that $k_i 
\hookrightarrow K$ for all $i$ (e.g. any residue field of $k_1 \otimes_k \ldots \otimes_k k_n$,
where $k$ is the (common) prime field; cf. \cite{Bo}, Section V.2.4, Cor. to Prop. 4). There is 
a ring map $\varphi : R \to M_n(K)$, sending $(r_1, \ldots, r_n) \in 
k_1 \times \ldots \times k_n \mapsto \diag(r_1, \ldots, r_n) \in M_n(K)$.
Then $r \in R$ is a unit iff $\varphi(r)$ is a unit in $M_n(K)$, so applying Theorem 1 
to $\varphi(a_1), \ldots, \varphi(a_m) \in M_n(K)$ gives the result.
\end{proof}

\begin{rem}
If $n = 2$, then Theorem 6 holds without assuming that $\ch R/\m_1 = \ch R/\m_2$:
if $a_1, a_2 \in k_1 \times k_2$ are such that none of $a_1, a_2, 1 - a_1, 1 - a_2$ is a unit, 
then $a_1, a_2 \in \{(1,0), (0,1)\}$, so either $a_1 + a_2$ or $1 - (a_1 + a_2)$ is a unit.

In general though, the hypothesis of equal characteristics in Theorem 6 is crucial, as the
following examples show: 
\end{rem}

\begin{ex}
In $\F_2 \times \F_3 \times k$ (where $k$ is any field), the elements 
\[a_1 = (0, 1, 1), \]
\[a_2 = a_3 = (1, -1, 0), \]
\[a_4 = 1 - (a_1 + a_2 + a_3) \]
satisfy $a_1 + a_2 + a_3 + a_4 = 1$, but no subset of $\{a_1, a_2, a_3, a_4\}$ of size 
$\le 3$ sums to a unit (to mentally verify this, it suffices to check that any nonempty subsum 
of $\{a_1, a_2, a_3\}$ contains both a coordinate equal to $0$ and a coordinate equal to $1$).

There are also (many) such examples with all elements $a_i$ distinct: in 
$\F_2 \times \F_3 \times k_1 \times k_2$ (where $k_1, k_2$ are any fields), the elements
\[ a_1 = (0,0, 0,1), \]
\[ a_2, a_3, a_4 = (1,-1,0,*), \]
\[ a_5 = 1 - (a_1+a_2+a_3+a_4) = (0,1,1, \cdot) \]
also sum to a unit, although no proper subset of them does. 
Here each $*$ can be taken to be any element in $k_2$, so if $|k_2| > 2$, then the 
$a_i$ can be chosen to be pairwise distinct.
\end{ex}

Finally, we record some additional interesting consequences of Lemma 2. 
The key feature of the next proposition is nonemptyness of the subset $S$:

\begin{prop}
Let $R$ be a ring, $n \in \N$, and pick any $A_1, \ldots, A_n \in M_n(R)$. 
Then for any $B \in M_n(R)$ with $\det B \ne 0$, there exists $\emptyset \ne S
\subseteq [n]$ such that $\det \left( \bigg( \sum_{i \in S} A_i \bigg) + B \right) - 
\det \bigg( \sum_{i \in S} A_i \bigg) \ne 0$.
\end{prop}

\begin{proof}
This follows from the identity
\[
\sum_{\emptyset \ne S \subseteq [n]} (-1)^{|S|} \left( \det \bigg( \sum_{i \in S} 
A_i \bigg) - \det \left( \bigg( \sum_{i \in S} A_i \bigg) + B \right) \right) = \det B
\]
which results from applying Lemma 2 to 
$A_1, \ldots, A_n, B$.
\end{proof}

We end with a geometric interpretation, in terms of additively generated point 
configurations on cones over projective hypersurfaces:

\begin{prop}
Let $k$ be a field, $\operatorname{char} k = 0$, $n \in \N$, and let $X = V({\det}_n) 
\subseteq \mathbb{A}^{n^2}_k$ be the affine cone over the determinantal hypersurface. 
Let $\Delta \subseteq \mathbb{A}^{n^2}$ be an $n$-simplex with first barycentric 
subdivision $\overline{\Delta}$. If all vertices of $\overline{\Delta}$ other than the 
centroid lie on $X$, then in fact the centroid of $\overline{\Delta}$ also lies on $X$.
\end{prop}

\begin{proof}
Let $p_0, \ldots, p_n$ be the vertices of $\Delta$. Each vertex of $\overline{\Delta}$
is of the form $\frac{1}{|S|} \bigg( \sum_{i \in S} p_i \bigg)$ for some $\emptyset 
\ne S \subseteq [n]$, which (by homogeneity) lies on $X$ iff $\sum_{i \in S} p_i$ 
does. Viewing each $p_i$ as an $n \times n$ matrix over $k$ and applying 
Lemma 2 gives the result.
\end{proof}

In fact, the proof of \cite{FBCC} shows that Lemma 2 holds for any 
homogeneous polynomial, so Proposition 8 actually holds for any cone $X$ (in 
any affine space) cut out by a degree $n$ polynomial (and thus also for any 
intersection of such cones). 

\vskip 2ex

\noindent \textbf{Acknowledgements:} The author would like to thank Joe Kileel
for valuable discussions, in particular for contributing to the proof of Lemma 3;
as well as David Eisenbud for helpful comments. The examples following Theorems 5 
and 6 were carried out with the help of Macaulay2.


\medskip

\begin{thebibliography}{1}

\bibitem{SA}
S. A. Amitsur.
\newblock {\em On the Characteristic Polynomial of a Sum of Matrices}.
\newblock Linear and Multilinear Algebra 8(3), p. 177--182, 1980.

\bibitem{Bo}
N. Bourbaki.
\newblock {\em \'{E}l\'ements de math\'ematique. {A}lg\`ebre, {c}hapitres 4 \`a
  7}.
\newblock Masson, Paris, 1981.

\bibitem{CS}
R. S. Costas-Santos.
\newblock {\em On the elementary symmetric functions of a sum of matrices}.
\newblock Journal of Algebra, Number Theory: Advances and Applications 
1(2), p. 99--112, 2009.

\bibitem{FBCC}
A. Facchini, F. Barioli, R. Chapman, and R. M. Carroll.
\newblock {\em Nonsingular Sums of Matrices: 10784}.
\newblock American Mathematical Monthly 109(7), p. 664--666, 2002.

\bibitem{M2}
D. R. Grayson and M. E. Stillman.
\newblock {\em Macaulay2}, a software system for research in algebraic geometry.
\newblock Available at \url{http://www.math.uiuc.edu/Macaulay2/}.

\bibitem{RS}
C. Reutenauer and M. Sch\"utzenberger.
\newblock {\em A Formula for the Determinant of a Sum of Matrices}.
\newblock Letters in Mathematical Physics 13, p. 299--302, 1987.

\end{thebibliography}
\end{document}